\documentclass{amsart}

\usepackage{graphicx}        
\usepackage[bottom]{footmisc}
\usepackage{url}
\usepackage[utf8]{inputenc}
\usepackage{amsmath}
\usepackage{amsthm}
\usepackage{amssymb}
\usepackage{physics}

\newtheorem{assumption}{Assumption}

\newtheorem{theorem}{Theorem}
\newtheorem{lemma}[theorem]{Lemma}

\newtheorem{remark}[theorem]{Remark}

\newtheorem{definition}[theorem]{Definition}

\begin{document}

\title[Adaptive Numerical Methods for SDEwMS]{Adaptive Mesh Construction for the Numerical Solution of Stochastic Differential Equations with Markovian Switching}

\author{C\'onall Kelly}
\address{School of Mathematical Sciences, University College Cork, Western Gateway Building, Western Road, Cork, Ireland.}
\email{conall.kelly@ucc.ie}

\author{Kate O'Donovan}
\address{School of Mathematical Sciences, University College Cork, Western Gateway Building, Western Road, Cork, Ireland.}
\email{116370906@umail.ucc.ie}

%
%

\begin{abstract}
We demonstrate an approach to the numerical solution of nonlinear stochastic differential equations with Markovian switching. Such equations describe the stochastic dynamics of processes where the drift and diffusion coefficients are subject to random state changes according to a Markov chain with finite state space. We propose a variant of the Jump Adapted-Adaptive approach introduced by K, Lord, \& Sun~(2025) to construct nonuniform meshes for explicit numerical schemes that adjust timesteps locally to rapid changes in the numerical solution and which also incorporate the switching times of an underlying Markov chain as meshpoints. It is shown that a hybrid scheme using such a mesh that combines an efficient explicit method (to be used frequently) and a potentially inefficient backstop method (to be used occasionally) will display strong convergence in mean-square of order $\delta$ if both methods satisfy a mean-square consistency condition of the same order in the absence of switching. We demonstrate the construction of an order $\delta=1$ method of this type and apply it to generate empirical distributions of a nonlinear SDE model of telomere length in DNA replication.
\end{abstract}

\maketitle

\section{Introduction}
\label{sec:1}
We examine a method of constructing adaptive meshes for the efficient numerical solution of $d$-dimensional system of stochastic differential equations with Markovian switching (SDEwMS) of the form 
\begin{equation}
\begin{split}\label{GeneralSDEwMS}
    dX(t) &= f(X(t), r(t))dt + g(X(t), r(t))dW(t), \quad 0 \leq t \leq T,\\
    X(0) &= x_0\in\mathbb{R}^d.
\end{split}
\end{equation}
In \eqref{GeneralSDEwMS}, the drift and diffusion coefficients f and g take the form
\[
f: \mathbb{R}^d \times \mathbb{S} \rightarrow \mathbb{R}^d, \quad \text{and} \quad g: \mathbb{R}^d \times \mathbb{S} \rightarrow \mathbb{R}^{d \times m}.
\]
The stochastic dynamics of a process governed by \eqref{GeneralSDEwMS} are influenced by two independent sources of noise. The process $r$ dictates changes occurring at random intervals in the functional form of $f$ and $g$, and is described by a scalar continuous time Markov chain taking values in the set \(\mathbb{S}=\{1,2,..,L\}\). The process $W$ is an $m$-dimensional Brownian motion and captures the effect of a diffusion process on trajectories. We characterise both mathematically in Section \ref{sec:2}.

We are motivated by the following example where a scalar SDE is used to model the shortening over time of telomere length in DNA replication in jackdaws: 
\begin{equation}\label{eq:telomere}
dL(t)=-\left(c+aL(t)^2\right)dt+\sqrt{\frac{1}{3}aL(t)^3}dW(t),
\end{equation}
where $W$ is a scalar standard Brownian motion. As presented in \cite{grasman2011stochastic}, the model arises as the Langevin equation associated with the Fokker-Planck PDE 
\[
\frac{\partial p}{\partial t}=\frac{\partial}{\partial l}\left\{\left(c+\frac{1}{2}al^2\right)p\right\}+\frac{1}{2}\frac{\partial^2}{\partial l^2}\left\{\left(\frac{1}{3}al^3\right)p\right\},
\]
which governs the probability density function $p(t,l)$ for the telomere length $l$ at time $t$. By simulating an ensemble of trajectories of \eqref{eq:telomere} we can produce a density histogram of lengths $l=L$ at a fixed time $t=T$ which serves to numerically approximate $p(T,L)$. The parameters $c$ and $a$ control the underlying decay rate of the length and the intensity at which random breaks occur in the telomere respectively. 

The general form of the SDE \eqref{GeneralSDEwMS} without Markovian switching, can be written
\begin{equation}
\begin{split}\label{GeneralSDE}
    dX(t) &= f(X(t))dt + g(X(t))dW(t), \quad 0 \leq t \leq T,\\
    X(0) &= x_0\in\mathbb{R}^d.
\end{split}
\end{equation}
We can see that the telomere length SDE~\eqref{eq:telomere} is a special case of \eqref{GeneralSDE} with $d=m=1$, $f(x)=-(c+ax^2)$, and $g(x)=(ax^3/3)^{1/2}$.

Suppose we wish to solve \eqref{GeneralSDE} numerically over an arbitrary mesh $\{t_n\}=\{0=t_0,t_1,\ldots,t_{N-1},t_N=T$. Over any step $[t_n,t_{n+1}]$, solutions of \eqref{GeneralSDE} satisfy the integral equation
\begin{equation}\label{eq:SODE}
    X(s) = X(t_n) + \int_{t_n}^s f(X(r))dr + \int_{t_n}^s g(X(r))dW(r),\quad s\in[t_n,t_{n+1}].
\end{equation}
By fixing the solution at the beginning of the interval $[t_n,t_{n+1}]$, we achieve a continuous form of the explicit Euler-Maruyama approximation
\begin{equation} \label{eq:EMcont} 
    \bar{X} (s)  = X_n + \int_{t_n}^s f\left(\bar{X}(t_n)\right) dr + \int_{t_n}^s g\left(\bar{X}(t_n)\right) dW(r),\quad s\in[t_n,t_{n+1}].
\end{equation}
Setting $X_n:=\bar X(t_n)$, we can write the scheme in its computational form as 
\begin{equation}\label{eq:EMdiscrete}
    X_{n+1} = X_n + h_{n+1}f(X_n) + g(X_n)\Delta W_{n+1},\quad n=0,\ldots,N-1,
\end{equation}
where $h_{n+1}=t_{n+1}-t_n$ and $\Delta W_{n+1}=W(t_{n+1}) - W(t_n)$. Suppose there exists a constant maximum stepsize $h_{\max}\in(0,\infty)$ such that $h_n\leq h_{\max}$ for all $n=0,\ldots,N$. We say that solutions of \eqref{eq:EMcont} (or equivalently \eqref{eq:EMdiscrete}) converge strongly to solutions of \eqref{GeneralSDE} in $L_p$ with order $\gamma$ if and only if
\[
\max_{t\in[0,T]}\left(\mathbb{E}\left[\|X(t)-\bar X(t)\|^p\right]\right)^{1/p}\leq Ch_{\max}^{\gamma}.
\]
If $p=2$ then we say that \eqref{eq:EMcont} converges strongly in $L_2$, in mean-square. It is the case that the Euler-Maruyama method \eqref{eq:EMcont} converges in mean-square to \eqref{GeneralSDE} with order $\gamma = 1/2$ when f, g are globally Lipschitz and $h_n \equiv h$.

Hutzenthaler et al~\cite{hutzenthaler2011strong} confirmed that solutions of \eqref{eq:EMdiscrete} with constant stepsize $h_n\equiv h$ fails to converge strongly in the strong sense to solutions of \eqref{GeneralSDE} if either $f$ or $g$ grow superlinearly in norm. Indeed this is the case for the telomere model \eqref{eq:telomere}. To address this issue, many novel explicit methods have been proposed that converge strongly even for highly nonlinear drift and diffusion coefficients; these include tamed, projected, truncated, and adaptive variants of both Euler-Maruyama and Milstein methods (see, for example \cite{Beyn2017,Guo2018,HJK2012,mao2015truncated,wang2013tamed}). In particular we highlight the adaptive timestepping approach that was introduced for the Euler-Maruyama method in \cite{kelly2018adaptive} and extended to the Milstein method in \cite{KLS2023}, where it was used to construct a numerical method converging strongly with order $\gamma=1$ and to generate trajectories of the telomere SDE \eqref{eq:telomere}. 

We wish to extend this capability to models where the coefficients $f$ and $g$ undergo Markovian switching. For example in the case of the telomere SDE \eqref{eq:telomere} it may be that each of the model parameters in the pair $(c,a)$ undergo transition between two distinct values according to a Markov process. Alternatively one can view Markovian switching in the model parameters as a form of stress testing to investigate the robustness of sample distributions to random changes in these parameters.

We take as our template the approach to mesh construction in \cite{kelly2024strong}, where it was shown that an adaptive numerical scheme that satisfies a particular mean-square consistency bound can be extended to ensure strong convergence in mean-square for SDEs with jump perturbations (SJDEs) as long as the jump times are incorporated into the adaptive mesh in a particular way. In such a case the mesh is adapting both to the evolving behaviour of the approximate trajectory and directly to jumps when they occur. We refer to these as Jump Adapted-Adaptive Methods (JA-AMs).

In this article, we show that an analogue of the type of mesh used for JA-AMs also ensures convergence of explicit numerical schemes satisfying a similar mean-square consistency condition to SDEwMS governed by \eqref{GeneralSDEwMS}, as long as we include the times at which the Markov chain $r$ switches between distinct states. We prove strong convergence of such schemes in mean-square, noting that the major challenge of this analysis is dealing with the fact that the number of Markov transitions on any interval is finite but unbounded, and we handle it in the same manner as in the jump case \cite{kelly2024strong}. Finally we demonstrate the practical application of this approach on a variant of SDE model of telomere shortening \eqref{eq:telomere} with stochastic dynamics governed by coefficients that undergo transitions at random times.

In Section \ref{sec:2} we provide a complete mathematical description of the setting, including of $W$ and $r$, and provide conditions on $f$ and $g$ for which unique solutions of \eqref{GeneralSDEwMS} exist on $[0,T]$ for $T<\infty$. In Section 3 we describe how an adaptive mesh can be constructed that adapts both to the nonlinear functional response affecting the numerical solutions of \eqref{GeneralSDEwMS} and which incorporate the switching times of $r$. We describe the general form of our adaptive scheme and present a version of it based upon the Milstein method. In Section \ref{sec:4} we state and prove our man strong convergence theorem, which in particular confirms that the Milstein scheme applied over this class of adaptive meshes will converge strongly in $L_2$ with order $\delta=1$ to solutions of \eqref{GeneralSDEwMS}. Finally in Section \ref{sec:5} we describe the practical implementation of the scheme and demonstrate its use to generate sample histograms of telomere length in base pairs using a nonlinear SDEwMS model.

\section{Mathematical Preliminaries}
\label{sec:2}

In this section, we mathematically characterise the random processes $W$ and $r$ influencing the dynamics of \eqref{GeneralSDEwMS}, then specify additional constraints on the coefficients $f$ and $g$ that are necessary for our main result.

\subsection{Characterisation of random inputs}
First, let \((\Omega, \mathcal{F}, (\mathcal{F}_t)_{t\geq0}, \mathbb{P})\) be a complete probability space with a filtration \((\mathcal{F}_t)_{t\geq0}\) which is right continuous and where \(\mathcal{F}_0\) contains all $\mathbb{P}$-null sets. The process $W(t) = (W(t_1), . . .,W(t_m))^T$ is an $m$-dimensional Brownian motion defined on this probability space such that each $W_i$ is independent of $W_j$ for $i\neq j$, and each $(W_i)_{t\in[0,T]}$ is a scalar stochastic process with $W_i(0)=0$, a.s. continuous trajectories, and stationary increments that are independent on non-overlapping intervals.

Let \(r(t)\), \( t \geq 0 \) be a right-continuous Markov chain on the same probability space, taking values in \( \mathbb{S}\) with generator \(\Gamma = (\gamma_{ij})_{L \times L} \) given by

\begin{equation} \label{probrdelta}
\mathbb{P}[r(t+\Delta) = j|r(t) = i] = 
\begin{cases}
    \gamma_{ij}\Delta + o(\Delta), & \text{if $i$ $\neq$ $j$};    \\
    1+ \gamma_{ii}\Delta + o(\Delta), & \text{if $i$ = $j$};
\end{cases}
\end{equation}
where \(\Delta > 0\) and $o(\Delta)$ are higher order terms of $\Delta$. Here \(\gamma_{ij} \geq 0\) is the transition rate from $i$ to $j$ if \(i \neq j\) while 
\begin{equation} \label{gammaii}
\gamma_{ii} = - \sum_{j \neq i} \gamma_{ij}
\end{equation}
We assume that the standard Brownian Motion \(W(\cdot)\) and the Markov chain \(r(\cdot)\) are both \(\mathcal{F}_t\)-adapted, and for the latter the initial value $r(0)=r_0$ is an $\mathcal{F}_0$-measurable random variable taking values in $\mathbb{S}$. 

The generator matrix \(\Gamma\) for a state space $\mathbb{S} = \{1, . . ., L\}$ is 
\begin{equation} \label{generatorMatrix}
\Gamma = 
\begin{bmatrix}
    \gamma_{11} & \gamma_{12} & ... & \gamma_{1L}\\
    \gamma_{21} & \gamma_{22} & ... & \gamma_{2L}\\
    & &. \\
    & &. \\
    \gamma_{L1} & \gamma_{L2} & ... & \gamma_{LL}
\end{bmatrix}
\end{equation}
From \eqref{gammaii} we see that $\sum_{j=i}^{L} \gamma_{ij} =1$ for all $i = 1, ..., L$. 
Denote the $k^{th}$ time at which the Markov chain switches between distinct states as $\tau_k$. Then \((\tau_i)_{i\geq0}\) is an increasing sequence of non-negative random variables which can be written
\begin{equation}
    (\tau_i)_{i\geq0} = \sum_{k = 0}^i \pi_k,
\end{equation}
where $\pi_k$ denotes the waiting time between the $k^{th}$ and $(k+1)^{th}$ such transitions. In practice, as long as we keep track of the current state $i\in\mathbb{S}$ of the Markov chain, these waiting times may be computed as the holding times associated with each $i$, and these are modelled as exponentially distributed random variables with state-dependent rates \(\lambda_i\): see Section \ref{sec:5}. 

Associated with the Markov chain \(r(t)\) is a process \((\bar{N}_t)_{t\geq0}\) that counts the number of switches that take place over the interval \([0,T]\) for all \(t\geq0\), written 
\begin{equation}
     \bar{N}_t := \#\{i \geq 1, \tau_i \in [0,T]\}.
\end{equation}
Since the state space $\mathbb{S}$ is a finite set, $\bar{N}_t$ is almost surely (a.s) finite on any interval $[0,T]$: see for example \cite[Theorem 2.7.1]{Norris_1997}. However $\bar N_t(\omega)$ is not uniformly bounded from above across the sample space $\Omega$ for any  $t\in[0,T]$, and this presents a significant challenge in our main convergence result. A key to overcoming this challenge is the following.
\begin{assumption}\label{assum:Fdecom}
We suppose that \((\mathcal{F}_t)_{t\geq0}\) can be decomposed such that \(\mathcal{F}_t = \sigma\left(\mathcal{G}_t \cup \mathcal{H}_T \right) \) for all $t \geq 0$, where $W$ is adapted to the filtration $(\mathcal{G}_t)_{t\geq0}$, $r$ is adapted to the filtration $(\mathcal{H}_t)_{t \geq 0}$, and moreover the $\sigma$-algebra $\mathcal{G}_t$ is independent of $\mathcal{H}_T$ for all $t \in [0,T]$. 
\end{assumption}
Assumption \ref{assum:Fdecom} holds if \(r(\cdot)\) is independent of \(W(\cdot)\), and in practice it allows us to simulate trajectories of \eqref{GeneralSDEwMS} by precomputing trajectories of $r$ and incorporating switching times into an adaptive mesh.

\subsection{Minimal conditions on the drift and diffusion coefficients $f$ and $g$.}

We identify conditions on $f$ and $g$ that ensure the existence of unique solutions over each interval $[0,T]$. We require that both satisfy a local Lipschitz condition:
\begin{assumption}\label{assum:locLip}
    For every integer $k\geq 1$, there exists a positive constant $h_k$ such that, for all  \(i \in \mathbb{S}\) and those \(x, y \in \mathbb{R}^n\) with \(|x| \vee |y| \leq k\)
    \begin{equation}
        \|f(x, i)-f(y,i)\|^2 \vee \|g(x,i)-g(y,i)\|^2 \leq \bar{h}_k \|x-y\|^2,
    \end{equation}
\end{assumption} 
Consider also the following linear growth condition:
\begin{assumption}\label{assum:linGrow}
$f$ and $g$ satisfy a linear growth bound if for all \((x,i) \in \mathbb{R}^d \times \mathbb{S}\) there exists a constant \(0 < K_3< \infty\) such that 
\begin{equation}
    \|f(x, i)\|^2 \vee \|g(x, i)\|^2 \leq K_3 (1 + \|x\|^2).
\end{equation}
\end{assumption}
Under Assumptions \ref{assum:locLip} and \ref{assum:linGrow} there exists a unique solution to \eqref{GeneralSDEwMS} on each \([0,T]\): See Mao \& Yuan~\cite[Theorem 3.16]{mao2006stochastic}.

It is possible to replace the linear growth condition with a weaker constraint. If Assumption \ref{assum:locLip} holds, so that $f$ and $g$ satisfy a local Lipschitz condition, and they also satisfy the following monotonicity condition
\begin{assumption} \label{assump2.3}
    There exists a positive constant K such that for all $(x, i) \in \mathbb{R}^d \times [0, T] \times \mathbb{S}$,
    \begin{equation}
        x^T f(x,i) + \frac{1}{2} \|g(x,i)\|^2 \leq K(1 + \|x\|^2).
    \end{equation}    
\end{assumption}
Then there exists a unique solution $X$ to Eq. \eqref{GeneralSDEwMS} on each $[0,T]$. Under the same conditions (Assumption \ref{assum:locLip} and either Assumption \ref{assum:linGrow} or \ref{assump2.3}) solutions of \eqref{GeneralSDEwMS} also satisfy uniform moment bounds and H\"older regularity bounds: for details see~\cite{mao2006stochastic}. 

In our main strong convergence proof we will make use of the following continuous form of the Gronwall inequality (see, for example Mao~\cite[Theorem 8.1]{mao2007SDEapp}):
\begin{lemma} \label{Gronwall}
Let $T > 0$ and $c \geq 0$. Let $u$ be a Borel measurable bounded nonnegative function on $[0,T]$, and let $v$ be a nonnegative integrable function on ${0,T}$. If
\[
u(t) \leq c + \int_0^t v(s)u(s) ds \quad \text{for all} \quad 0 \leq t \leq T,
\]
then
\begin{equation}
u(t) \leq c \exp\Big( \int_0^t v(s) ds \Big) \quad \text{for all} \quad 0 \leq t \leq T.
\end{equation}

\end{lemma}

\section{Adaptive timestepping for an SDEwMS}
\label{sec:3}
In this section we define our adaptive mesh as a variant of the construction in \cite{kelly2024strong}, and characterise the Euler-Maruyama and Milstein methods method used to generate the approximate solutions for the class of SDEsWMS.

\subsection{Structure of an adaptive mesh}\label{sec:mesh}

Let \(\{h_{n+1}\}_{n \in \mathbb{N}}\) be a sequence of strictly positive random timesteps with corresponding random times \( \{t_n := \sum_{i=1}^n h_i\}_{n \in \mathbb{N}\setminus\{0\}}\), where \(t_0 = 0\). We set up a framework by which nodes in this mesh can be generated procedurally as a function of the solution at the start of each step, while also ensuring that the timepoints at which the Markov chain switches between distinct states are included. 

\begin{assumption} \label{assump3.1}
    For each term of the sequence of random timesteps \(\{h_{n+1}\}_{n\in\mathbb{N}}\) at least one of the following three possibilities holds:
    \begin{enumerate}
        \item there are constant values \(h_{\text{max}} > h_{\text{min}} > 0\), \(\rho > 1\) such that \(h_{\text{max}} = \rho h_{\text{min}}\), and
    \[
    0 < h_{\text{min}} \leq h_{n+1} \leq h_{\text{max}} \leq 1;
    \]
        \item there exists a switching time $\tau_k$, as defined in Section 2, such that $t_{n+1} = \tau_k$, so that 
        \[
        0 < h_{n+1} = \tau_k -t_n \leq h_{\max} \leq 1;
        \]
        \item $t_{n+1} = T$, so that
        \[
        0 < h_{n+1} = T - t_n \leq h_{\max} \leq 1.
        \]
    \end{enumerate}
In addition we assume each $h_{n+1}$ is $\mathcal{F}_{t_n}$-measurable.
\end{assumption}
It was shown in \cite{kelly2024strong} that as long as $h_1$ can be computed from the initial data, and successive adaptive timesteps $h_{k+1}$ are computed on each trajectory as a deterministic function of the numerical approximation at time $t_k$, then each $t_n$ on a mesh satisfying the conditions of Assumption \ref{assump3.1} is an $\mathcal{F}_t$-stopping time.

\begin{definition} \label{DefN(t)}
    Let \(N^{(t)}\) be a random integer such that 
    \begin{equation}
        N^{(t)} := \max\{n \in \mathbb{N}\ \{0\} : t_{n-1} < t\},
    \end{equation}
    and let \(N = N^{(t)}\) and \(t_N = T\), so that T is always the last point on the mesh
\end{definition}
Let Definition \ref{DefN(t)} be satisfied with \(N_{\min}^{(t)} := \lfloor t/h_{\max} \rfloor\) and \(N_{\max}^{(t)} := \lceil t/h_{\min} + \bar{N}_t \rceil \) 
We ensure that we reach the final time by taking \(h_N = T - t_{N-1}\) as our final step, and use the backstop method if \(h_N < h_{\min}\). Recall from Section \ref{sec:1} that \(\bar{N}_t < \infty\) a.s for all $t \in [0, T]$, where $\mathbb{S}$ is a finite set, therefore $N_{\max}^{(t)}$ is an a.s finite random variable.

\subsection{An explicit discretisation scheme for an SDEwMS}

While the main result in this article applies for any strongly convergent numerical method, in this section we give the example of the explicit Euler-Maruyama and Milstein methods, and show how may be adapted to numerically solve an SDEwMS. 

\subsubsection{The explicit Euler-Maruyama scheme}
We may extend \eqref{eq:EMdiscrete}, presented in Section \ref{sec:1}, to solve an SDEwMS over an adaptive mesh $\{t_n\}$ as
\begin{equation}
    X_{n+1} = X_n + h_{n+1} f(X_n, r(t_n)) + g(X_n, r(t_n))\Delta W_{n+1}.
\end{equation}
Then the continuous approximation for all \(s \in [t_n, t_{n+1}]\) is
\begin{equation}
    \bar{X} (s) = X_n + \int_{t_n}^s f(\bar{X}(t_n), r(t_n)) dr + \int_{t_n}^s g(\bar{X}(t_n), r(t_n)) dW(r).
\end{equation}
Notice that since switching times are included in the adaptive mesh as constructed in Section \ref{sec:mesh} we can simply fix the value of the Markov chain $r$ to the value observed at the beginning of each step. 

\subsubsection{The Milstein scheme}
In the case where $d=m=1$, a continuous form of the Milstein scheme for the SDE \eqref{eq:SODE} may be given as
\begin{equation}\label{eq:Mildiscrete}
    X_{n+1} = X_n + h_{n+1}f(X_n) + g(X_n)\Delta W_{n+1}+\frac{1}{2}g'(X_n)g(X_n)\left(\Delta W_{n+1}^2-h_{n+1}\right),
\end{equation}
with corresponding continuous version
\begin{multline}\label{eq:Milcont}
    \bar{X} (s) = X_n + \int_{t_n}^s f(\bar{X}(t_n)) dr
    \\+ \int_{t_n}^s g(\bar{X}(t_n)) dW(r)+\int_{t_n}^s\int_{t_n}^rg'(\bar{X}(t_n))g(\bar{X}(t_n))dW(p)dW(r).
\end{multline}

Again, we extend \eqref{eq:Mildiscrete} to solve an SDEwMS over an adaptive mesh $\{t_n\}$ as
\begin{multline}\label{eq:Mildiscreteadapt}
    X_{n+1} = X_n + h_{n+1}f(X_n, r(t_n)) + g(X_n, r(t_n))\Delta W_{n+1}
    \\+\frac{1}{2}g'(X_n, r(t_n))g(X_n, r(t_n))\left(\Delta W_{n+1}^2-h_{n+1}\right),
\end{multline}
with corresponding continuous version
\begin{multline}\label{eq:Milcontadapt}
    \bar{X} (s) = X_n + \int_{t_n}^s f(\bar{X}(t_n), r(t_n)) dr
    + \int_{t_n}^s g(\bar{X}(t_n), r(t_n)) dW(r)
    \\+\int_{t_n}^s\int_{t_n}^rg'(\bar{X}(t_n), r(t_n))g(\bar{X}(t_n), r(t_n))dW(p)dW(r).
\end{multline}
A full description of the Milstein scheme applied to \eqref{eq:SODE} in the more general $d$-dimensional case with $m$ independent Brownian noise terms requires the simulation of L\'evy areas. Details are in \cite{KLSBit}. When $m=d=1$, the Milstein method \eqref{eq:Milcont} may be expressed as a map $\psi\,:\,\mathbb{R}\times\mathbb{S}\times[0,T]\times[0,h_{\max}]\to\mathbb{R}$ over each step $[t_n,t_{n+1}]$:
\begin{multline}\label{eq:MilMap}
\psi(x,i,t_n,s-t_n):=x+(s-t_n)f(x,i)
\\+\int_{t_n}^{s}g(x,i)dW(v)+\int_{t_n}^{s}\int_{t_n}^{v}g'(x,i)g(x,i)dW(p)dW(v).
\end{multline}

Finally, note that if we modify the second term on the RHS of \eqref{eq:Mildiscrete} to be $h_{n+1}f(X_{n+1})$ we arrive at the implicit Milstein scheme, which will be used later in Section \ref{sec:5}. A strong convergence analysis of this scheme for SDEs with non-globally Lipschitz continuous coefficients is contained in \cite{Desmond}.

\subsection{A hybrid adaptive numerical scheme: general form}
Constructing our scheme we will use two maps $\mathcal{M}: \mathbb{R}^d\times\mathbb{S} \times [0,T] \times [0,h_{\text{max}}]\to\mathbb{R}^d$ and $\mathcal{B}: \mathbb{R}^d\times\mathbb{S} \times [0,T] \times [0,h_{\text{max}}]\to\mathbb{R}^d$ where $\mathcal{M}$ is the main or default map and $\mathcal{B}$ the backstop map. \(\mathcal{M}\) corresponds to an efficient scheme with is convergent as an adaptive method and \(\mathcal{B}\) corresponds to a method which is convergent as a fixed step method but which may be inefficient or may induce distortions in solution dynamics, with the aim that $\mathcal{B}$ will be rarely used. 

We characterise $\mathcal{M}$ and $\mathcal{B}$ through their action over a single step, requiring them to satisfy a certain mean-square consistency bound.

\begin{definition} \label{DefYmap}
    Let $ \{h_{n+1}\}_{n \in \mathbb{N}}$ satisfy Assumption \eqref{assump3.1}. We define the continuous form of hybrid adaptive numerical scheme associated with the timestepping strategy \( \{h_{n+1}\}_{n \in \mathbb{N}} \) to be
    \begin{multline} \label{Y(s)}
        Y(s) = \mathcal{M}( Y(t_n),r(t_n), t_n, s - t_n) \cdot \mathcal{I}_{\{h_{\text{min}} < h_{n+1} \leq h_{\text{max}}\}} \\
         + \mathcal{B}(Y(t_n),r(t_n), t_n, s - t_n) \cdot \mathcal{I}_{\{h_{n+1} \leq h_{\text{min}}\}}.
    \end{multline}
\end{definition}

Now define the global error of the scheme
\begin{definition}
    For any \(s \in [0,T]\), we define the global error acquired by the scheme \eqref{Y(s)} to be
    \begin{equation}
        E(s) : = X(s) - Y(s)
    \end{equation}
    where $X(s)$ satisfies \eqref{GeneralSDEwMS} and $Y(s)$ satisfies \eqref{Y(s)}.
\end{definition}

Then we make the following assumption
\begin{assumption} \label{assump3.2}
    For $s \in [t_n, t_{n+1}]$, $n \in \mathbb{N}$ and any $x, y\in \mathbb{R}^d$, let
    \begin{equation} 
        \begin{split}
            &E_{\mathcal{M}}^{x,y,i} (s) :=  x - \mathcal{M}(y,i, t_n, s - t_n); \\
            &E_{\mathcal{B}}^{x,y,i} (s) :=  x - \mathcal{B}(y,i, t_n, s - t_n).     
        \end{split}
    \end{equation}
        Assume that $E_{\mathcal{M}}^{x,y,i}$ and $E_{\mathcal{B}}^{x,y,i}$ satisfy, for some $\delta > 0$, both $\psi = \mathcal{M} $ or $\psi = \mathcal{B} $, any $i\in\mathbb{S}$, and any a.s. finite \( \mathcal{F}_{t_n}\)-measurable $\mathbb{R}^d$-valued random variables \(A_n, B_n\) 
    \begin{multline} \label{mapError}
        \mathbb{E} \bigr[ || E_{\psi}^{A_n, B_n,i} (t_{n+1}) ||^2 | \mathcal{F}_{t_n} \bigr] \leq || A_n - B_n ||^2 \\
        + \Gamma_1 \int_{t_n}^{t_{n+1}} \mathbb{E} \bigr[ ||E_{\psi}^{A_n, B_n,i} (r)||^2 | \mathcal{F}_{t_n} \bigr] dr + \Gamma_{2, n} h_{n+1}^{2\delta +1}.
    \end{multline}
    In \eqref{mapError}, for any $h_{n+1} \leq h_{\min}$ in the case where $\psi = \mathcal{B}$ and for any $h_{\min} < h_{n+1} \leq h_{\max}$ in the case where $\psi = \mathcal{M}$, $\Gamma_1 < \infty $ is a constant, $\Gamma_{2,n}$ is a scalar $\mathcal{F}_{t_n}$-measurable random variable with finite expectation denoted $\Gamma_2$, and both $\Gamma_1$ and $\Gamma_{2,n}$ are independent of $h_{\max}$.
\end{assumption}

\subsection{A hybrid adaptive numerical scheme: particular form}
To ensure the appropriate bound to satisfy Assumption \eqref{assump3.2} we use a particular kind of adaptive time-stepping strategy such that, whenever \( h_{\text{min}} < h_{n+1} \leq h_{\text{max}}\),
\begin{equation}\label{eq:Ybound}
    || Y(t_n) || < R.
\end{equation}

\begin{lemma}
    Let $Y(s)$ be the hybrid scheme \eqref{Y(s)}. Fix \(n = 0,...,N-1\) and \(k >0\). Then, in the event that \( h_{\text{min}} < h_{n+1} \leq h_{\text{max}}\), $h_{n+1}$ satisfies \eqref{eq:Ybound} 
    \begin{equation}\label{eq:hscheme}
        h_{n+1} = \left(h_{\text{min}} \vee \left(\frac{h_{\text{max}}}{||Y(t_n)||^{1/k}} \wedge h_{\text{max}} \right) \right) \wedge \left(\tau_{\bar{N}_{t_n} +1} - t_n \right)
    \end{equation}
\end{lemma}
For a proof, see Lemma 4.2 in \cite{kelly2024strong}. Since the mesh $\{t_n\}$ has been constructed so that the value of the Markov chain $r$ cannot change between meshpoints, conditions under which the Milstein map \eqref{eq:MilMap} satisfies \eqref{mapError} over a single step with $\delta =1$ are independent of $i$ and the same as those provided in \cite[Section 4]{kelly2024strong}.

\section{Main result}
For any adaptive numerical scheme composed of main and backstop maps that satisfy the bound \eqref{mapError} in Assumption \ref{assump3.2}, our main result demonstrates strong $L_2$ convergence when extended to an SDEwMS via a switching-time adapted mesh. The proof is a modification of that of \cite[Theorem 4.1]{kelly2024strong}, though the change of setting from SJDEs to SDEwMS allows for the analysis to be simplified in places. By using the Milstein method \eqref{eq:Milcont} we can construct a specific scheme of order $\delta=1$.
\label{sec:4}
\begin{theorem}  
Let $(X(t))_{t\in[0,T]}$ be a solution of \eqref{GeneralSDEwMS} with initial value $X(0) = x_0 \in \mathbb{R}^d$ and suppose that Assumptions \ref{assum:Fdecom}, \ref{assum:locLip}, and \ref{assump2.3} hold. Let $Y$ be an adaptive numerical scheme as characterised in Definition \ref{DefYmap} with $Y_0 = x_0$ such that Assumption \ref{assump3.2} holds with some $\delta > 0$. Then there exists a constant $C(T)>0$ such that
    \begin{equation}\label{eq:mainErrorBound}
        \max_{t \in [0,T]} \Big( \mathbb{E}\Big[ \norm{X(t)-Y(t)}^2\Big] \Big)^{\frac{1}{2}} \leq C(T)h_{\max}^{\delta},
    \end{equation}
where C does not depend on $h_{\max}$.
\end{theorem}

\begin{proof}
The proof naturally divides into four steps which successively show how conditional mean-square errors accumulate, first over  the steps between successive switching times, then from the initial time to the final switching time before a fixed time $t$, then from the last switching time to $t$, and finally over the full interval $[0,t]$. We may then apply the Gronwall inequality (see Lemma \ref{Gronwall}) to complete the proof.

 \medskip
 \noindent {\bf Step 1: cumulative error bound over the interval between successive switches.} Over a single step, we obtain from \eqref{mapError} in Assumption \ref{assump3.2} the following error bound regardless of whether the scheme uses the main map or the backstop map over a single step of length $0<h_{n+1}\leq h_{\max}$, and regardless of whether or not a change in the state of the Markov chain $r$ occurs at either endpoint of the timestep: a.s,
 \begin{equation} \label{OneStepError}
     \mathbb{E}\left[\norm{E (t_{n+1})}^2 | \mathcal{F}_{t_n}\right] \leq \norm{E(t_n)}^2 
     + \Gamma_1 \int_{t_n}^{t_{n+1}}\mathbb{E}\left[\norm{E (u)}^2 | \mathcal{F}_{t_n}\right]du
      +\Gamma_{2,n}h_{n+1}^{2\delta+1}.
 \end{equation}
Let $\tau_k$ and $\tau_{k+1}$ be successive switching times. Subtract the first term on the RHS of \eqref{OneStepError} from both sides of the inequality and sum both sides over all the steps from $\tau_k$ to $\tau_{k+1}$. Then we get
\begin{eqnarray*} 
    \mathcal{K}_{k} &:=& \sum_{n=N^{(\tau_k)}}^{N^{(\tau_{k+1})}-1} \bigg( \mathbb{E}\left[\norm{E(t_{n+1}}^2 | \mathcal{F}_{t_n}\right] - \norm{E(t_n)}^2 \bigg)\mathcal{I}_{\{N^{(\tau_{k+1})}>n\}}\\
    &\leq& \sum_{n=N^{(\tau_k)}}^{N^{(\tau_{k+1})}-1} \bigg(\Gamma_1 \int_{t_n}^{t_{n+1}}\mathbb{E}[\norm{E(u)}^2 | \mathcal{F}_{t_n}]du 
     +\Gamma_{2,n}h_{n+1}^{2\delta+1} \bigg)\mathcal{I}_{\{N^{(\tau_{k+1})}>n\}} \\
     &=& \Gamma_1 \sum_{n=N^{(\tau_k)}}^{N^{(\tau_{k+1})}-1} \int_{t_n}^{t_{n+1}}\mathbb{E}\left[\norm{E(u)}^2 | \mathcal{F}_{t_n}\right]\mathcal{I}_{\{N^{(\tau_{k+1})}>n\}}du
     \\&&+ \sum_{n=N^{(\tau_k)}}^{N^{(\tau_{k+1})}-1}\Gamma_{2,n} h_{n+1}^{2\delta+1}\mathcal{I}_{\{N^{(\tau_{k+1})}>n\}},\quad a.s.
\end{eqnarray*}
By Definition \ref{DefN(t)}, the filtration $\mathcal{F}_{t_n}$ may be written as $\mathcal{F}_{t_{N^{(u)}-1}}$, for all $u \in [t_n, t_{n+1})$ and $n=0,\ldots,N^{(u)}-1$, so we can rewrite the inequality as
\begin{equation} \label{KkGamma1}
    \mathcal{K}_{k} \leq \Gamma_1 \int_{\tau_k}^{\tau_{k+1}}\mathbb{E}\left[\norm{E(u)}^2 \big|\mathcal{F}_{t_{N^{(u)}-1}}\right]dr + \sum_{n=N^{(\tau_{k})}}^{N^{(\tau_{k+1})}-1}\Gamma_{2,n} h_{n+1}^{2\delta +1},\quad a.s,
\end{equation}
where all indicator variables on the RHS have been bounded above by $1$.

\medskip
\noindent{\bf Step 2: cumulative error from time $t=0$ to the time of the last switch before time $t$.}
Fix any $t\in[0,T]$. Summing over $k$ for all switching times on the interval $(0, t]$ and denoting $\mathcal{Q}(t):=\sum_{k=0}^{\bar{N}_t -1} \mathcal{K}_{k}$, we have a.s,
\begin{equation}
    \mathcal{Q}(t) \leq \Gamma_1 \sum_{k=0}^{\bar{N}_t -1} \int_{\tau_k}^{\tau_{k+1}}\mathbb{E}\left[\norm{E(u)}^2 \big|\mathcal{F}_{t_{N^{(u)}-1}}\right]du +\sum_{k=0}^{\bar{N}_t -1}\sum_{n=N^{(\tau_{k})}}^{N^{(\tau_{k+1})}-1}
    \Gamma_{2,n} h_{n+1}^{2\delta + 1}.
\end{equation}
Set $\tau_0 = 0$ to include time from $0$ to the first switch $\tau_1$ and take the expectation conditional upon $\mathcal{H}_T$
\begin{multline*}
    \mathbb{E}\left[\mathcal{Q}(t) | \mathcal{H}_T\right] \leq \Gamma_1 \mathbb{E} \left[ \sum_{k=0}^{\bar{N}_t -1} \int_{\tau_k}^{\tau_{k+1}}\mathbb{E}\left[\norm{E(u)}^2 \big|\mathcal{F}_{t_{N^{(u)}-1}}\right]du \bigg|\mathcal{H}_T \right] \\
    + \mathbb{E} \left[\sum_{k=0}^{\bar{N}_t -1}\sum_{n=N^{(\tau_{k})}}^{N^{(\tau_{k+1})}-1} \Gamma_{2,n} h_{n+1}\bigg| \mathcal{H}_T \right] h_{\max}^{2\delta},\quad a.s.
\end{multline*}
Since $\bar{N}_t$ is $\mathcal{H}_T$ measurable we can take all of the summations on the RHS out of the conditional expectations. 
\begin{multline*}
    \mathbb{E}\left[\mathcal{Q}(t) | \mathcal{H}_T\right] 
    \leq \Gamma_1\sum_{k=0}^{\bar{N}_t -1}\mathbb{E}\left[ \int_{\tau_k}^{\tau_{k+1}}\mathbb{E}\left[\norm{E(u)}^2 \big|\mathcal{F}_{t_{N^{(u)}-1}}\right]du \big|\mathcal{H}_T \right] \\
    + \sum_{k=0}^{\bar{N}_t -1} \mathbb{E} \left[\sum_{n=N^{(\tau_{k})}}^{N^{(\tau_{k+1})}-1} \Gamma_{2,n} h_{n+1}\bigg| \mathcal{H}_T \right] h_{\max}^{2\delta},\quad a.s.
\end{multline*}

Since $\mathcal{H}_T \subseteq \mathcal{F}_{t_{N^{(u)}-1}}$ for any $u \in [0, T]$, an application of the tower property of conditional expectations yields 
\begin{multline} 
    \mathbb{E}\left[\mathcal{Q}(t) | \mathcal{H}_T\right]
    \leq \Gamma_1  \int_{0}^{\tau_{\bar{N}_t}}\mathbb{E}\left[\norm{E(u)}^2 |\mathcal{H}_T\right]du\\ + \sum_{k=0}^{\bar{N}_t -1} \mathbb{E} \left[\sum_{n=N^{(\tau_{k})}}^{N^{(\tau_{k+1})}-1} \Gamma_{2,n} h_{n+1}\bigg| \mathcal{H}_T \right] h_{\max}^{2\delta},\quad a.s.\label{zerotoLastSwitchError}
\end{multline}

\medskip
\noindent{\bf Step 3: error bound from the time of the last switch to time $t$.}
Fix $t\in[0,T]$. If $t_{N^{(t)}}=t$ this step can be omitted. Suppose that $t_{N^{(t)}}<t$. Replacing $t_n$ with $t_{N^{(t)}-1}$ and $t_{n+1}$ with $t$ in \eqref{OneStepError}, we have a bound on the error over the last step to time $t$ as
\begin{multline} \label{step4eq1}
        \mathbb{E}\left[\norm{E (t)}^2 | \mathcal{F}_{t_{N^{(t)}-1}}\right]
         \leq \norm{E(t_{N^{(t)-1}})}^2\\ + \Gamma_1 \int_{t_{N^{(t)}-1}}^{t}\mathbb{E}[\norm{E (u)}^2 | \mathcal{F}_{t_{N^{(t)}-1}}]du 
     +\Gamma_{2,n}|t-t_{N^{(t)}-1}|^{2\delta+1}.
\end{multline}
Now summing up to $t$ with the error bound over the last step satisfying \eqref{step4eq1} and taking the expectation conditioned on $\mathcal{H}_T$ we have, for all $t\in[0,T]$ and a.s,
\begin{eqnarray} 
       \mathcal{R}(t)&:=& \mathbb{E}\left[
       \sum_{n = N^{({\tau_{\bar{N}_{t}}})}}^{N^{(t)}-2}\left(
       \mathbb{E}\left[\norm{E (t_{n+1})}^2 | \mathcal{F}_{t_n}\right] - \norm{E(t_n)}^2\right)\mathcal{I}_{\{N^{(t)}>n+1\}}\right.\nonumber
       \\&&\left.\qquad\qquad\qquad\qquad+ \mathbb{E} \left[ 
       \norm{E (t)}^2 \big| \mathcal{F}_{t_{N^{(t)}-1}}
       \right] - \norm{E(t_{N^{(t)}-1})}^2
       \bigg| \mathcal{H}_T \right]\nonumber
       \\
        &\leq& \Gamma_1 \mathbb{E}\left[
        \sum_{n = N^{({\tau_{\bar{N}_{t}}})}}^{N^{(t)}-2} \int_{t_n}^{t_{n+1}}  \mathbb{E} \left[ \norm{E(u)}^2 | \mathcal{F}_{t_n}  \right]\mathcal{I}_{\{N^{(t)}>n+1\}} du\right.
        \\&&\left.\qquad+\int_{t_{N^{(t)}-1}}^{t}\mathbb{E}\left[\norm{E (u)}^2 | \mathcal{F}_{t_{N^{(t)}-1}}\right]du         \bigg|\mathcal{H}_T \right] \nonumber\\
        &&\qquad\qquad+ h_{\max}^{2\delta} \mathbb{E} \left[
        \sum_{n = N^{({\tau_{\bar{N}_{t}}})}}^{N^{(t)}-2}
        \Gamma_{2,n} h_{n+1}
        +\Gamma_{2,N^{(t)-1}}|t-t_{N^{(t)}-1}|
         \bigg|\mathcal{H}_T   \right].\label{eq:S4final}
\end{eqnarray}

Summing the integrals on the RHS of \eqref{eq:S4final}, bounding the indicators on the RHS above by $1$, and again applying the tower property of conditional expectations we get
\begin{multline} \label{lastJumptoLittlet}
       \mathcal{R}(t) 
        \leq  \Gamma_1  \int_{\tau_{\bar{N}_t}}^{t} \mathbb{E} \left[ \norm{E(u)}^2 | \mathcal{H}_T  \right] du \\
        + h_{\max}^{2\delta} \mathbb{E} \left[
        \sum_{n = N^{({\tau_{\bar{N}_{t}}})}}^{N^{(t)}-2}
        \Gamma_{2,n} h_{n+1}
        +\Gamma_{2,N^{(t)-1}}|t-t_{N^{(t)}-1}|
         \bigg|\mathcal{H}_T   \right],\, t\in[0,T],\,a.s.
\end{multline}

\medskip
\noindent{\bf Step 4: cumulative error bound from time $t=0$ to time $t$.}
Adding the error bound associated with the steps that occur after the last switch \eqref{lastJumptoLittlet} to the error bound at the last switch \eqref{zerotoLastSwitchError} we have a.s,

\begin{multline} \label{0toLittletError}
        \mathbb{E}\left[\mathcal{Q}(t) | \mathcal{H}_T \right] + \mathcal{R}(t) \leq \Gamma_1 \int_{0}^{t} \mathbb{E} \left[ \norm{E(u)}^2 \big| \mathcal{H}_T \right] du \\
        + h_{\max}^{2\delta} \mathbb{E} \left[ \sum_{n=0}^{N^{(t)}-2} \Gamma_{2,n} h_{n+1} + \Gamma_{2,N^{(t)-1}} |t - t_{N^{(t)}-1}| \bigg| \mathcal{H}_T \right],\quad t\in[0,T].
\end{multline}

Next lets simplify the LHS of \eqref{0toLittletError}:
\begin{eqnarray*} 
        \lefteqn{\mathbb{E}\left[\mathcal{Q}(t) \bigg| \mathcal{H}_T \right] + \mathcal{R}(t)} 
        \\ &=& \mathbb{E} \left[ \sum_{k=0}^{\bar{N}_t -1} \left(   \sum_{n = N^{(\tau_k)}}^{N^{(\tau_k)}-1} \left( \mathbb{E} \left[ \norm{E(t_{n+1})}^2 \Big| \mathcal{F}_{t_{n}} \right]  - \norm{E(t_n)}^2 \right)\mathcal{I}_{\{N^{(\tau_{k+1})}>n\}}\right.\right. \\
        &&\qquad\qquad\quad\left.\left.+ \sum_{n= N^{(\tau_{\bar{N}_t})}}^{N^{(t)}-2} \left( \mathbb{E}\left[ \norm{E(t_{n+1})}^2 \Big| \mathcal{F}_{t_{n}} \right]  - \norm{E(t_n)}^2 \right)\mathcal{I}_{\{N^{(t)}>n+1\}}
        \right)\right.  \\
        &&\qquad\qquad\qquad\qquad\qquad\qquad\left.+ \mathbb{E} \left[ \norm{E(t)}^2 \Big| \mathcal{F}_{t_{N^{(t)}-1}} \right] - \norm{E(t_{N^{(t)}-1})}^2
        \bigg| \mathcal{H}_T \right] \\
        &=& \mathbb{E} \left[ \sum_{n=0}^{N^{(t)}-2} \left( \mathbb{E} \left[ \norm{E(t_{n+1})}^2 \Big| \mathcal{F}_{t_{n}} \right]  - \norm{E(t_n)}^2 \right)\mathcal{I}_{\{N^{(t)}>n+1\}}\right. \\
        &&\qquad\qquad\qquad\left.+ \mathbb{E} \left[ \norm{E(t)}^2 \Big| \mathcal{F}_{t_{N^{(t)}}-1} \right] - \norm{E(t_{N^{(t)}-1})}^2 \bigg| \mathcal{H}_T \right],\,t\in[0,T],\,a.s.
\end{eqnarray*}

As $N_{\max}^{(t)}=\lceil t/h_{\min}+\bar N_t\rceil$ is $\mathcal{H}_T$ measurable, we can bound $N^{(t)}$ by $N_{\max}^{(t)}$ and move the sum out of the conditional expectation. Again applying the tower property and now the telescoping sum with $E(0)= 0$ we see that
\begin{eqnarray*} 
        \lefteqn{\mathbb{E}\left[\mathcal{Q}(t) | \mathcal{H}_T \right] + \mathcal{R}(t)}\nonumber \\
        &=& \sum_{n=0}^{N_{\max}^{(t)}-2} \mathbb{E}\left[  \mathbb{E} \left[ \norm{E(t_{n+1})}^2\mathcal{I}_{\{N^{(t)}>n+1\}} \big| \mathcal{F}_{t_{n}} \right]  - \norm{E(t_n)}^2\mathcal{I}_{\{N^{(t)}>n+1\}}\right.\nonumber  \\
        && \qquad\qquad\qquad\qquad\left.+ \mathbb{E} \left[ \norm{E(t)}^2 \Big| \mathcal{F}_{t_{N^{(t)}}-1} \right] - \norm{E(t_{N^{(t)}-1})}^2 \Big| \mathcal{H}_t \right] \nonumber\\
        &=& \mathbb{E}\left[ \norm{E(t)}^2 \big| \mathcal{H}_T \right],\, t\in[0,T],\,a.s.\label{LHSstep5}
\end{eqnarray*}
Substituting \eqref{LHSstep5} into \eqref{0toLittletError} we have
\begin{multline} \label{0toTGamma1}
        \mathbb{E}\left[ \norm{E(t)}^2 \big| \mathcal{H}_T \right] \leq \Gamma_1 \int_{0}^{t} \mathbb{E} \left[ \norm{E(u)}^2 \big| \mathcal{H}_T \right] du \\
        + h_{\max}^{2\delta} \mathbb{E} \left[ \sum_{n=0}^{N^{(t)}-2} \Gamma_{2,n} h_{n+1} + \Gamma_{2,N^{(t)-1}} |t - t_{N^{(t)}-1}| \bigg| \mathcal{H}_T \right],\,t\in[0,T],\,a.s.
\end{multline}

Now let's consider the second term on the RHS of \eqref{0toTGamma1}. On each trajectory $\omega \in \Omega$ we can define a step function $\bar{\Gamma}_2(s)(\omega)$, $s \in [t_{n}(\omega), t_{n+1}(\omega))$, $n = 0, \cdots, N^{(t)}(\omega)$.
Then write the summation as an integral and take the expectation to get
\begin{multline*}
        \mathbb{E}\left[ h_{\max}^{2\delta} \mathbb{E} \left[ \sum_{n=0}^{N^{(t)}-2} \Gamma_{2,n} h_{n+1} + \Gamma_{2,N^{(t)-1}} \left|t - t_{N^{(t)}-1}\right| \bigg| \mathcal{H}_T \right] \right] \\
        =h_{\max}^{2\delta} \mathbb{E}\bigg[ \mathbb{E} \bigg[ \int_{0}^{t} \bar{\Gamma}_{2}(s) ds
        \bigg| \mathcal{H}_T \bigg] \bigg].
\end{multline*}
Since $\mathbb{E}[\Gamma_{2,n}] < \Gamma_2 < \infty$ and by construction of the step function we see that
\begin{equation}\label{Step5Gamma2}
        h_{\max}^{2\delta} \mathbb{E}\left[ \mathbb{E} \left[ \int_{0}^{t} \bar{\Gamma}_{2}(s) ds
        \bigg| \mathcal{H}_T \right] \right] 
        = h_{\max}^{2\delta} \left[ \int_{0}^{t} \mathbb{E} \left[ \bar{\Gamma}_{2}(s) \right] ds 
        \right] \leq h_{\max}^{2\delta}t\Gamma_{2}.
\end{equation}
Substitute \eqref{Step5Gamma2} back into \eqref{0toTGamma1} and take expectations on both sides to get
\begin{equation} 
        \mathbb{E}\left[\mathbb{E}\left[ \norm{E(t)}^2 \big| \mathcal{H}_T \right] \right] \leq \mathbb{E}\left[\Gamma_1 \int_{0}^{t} \mathbb{E} \left[ \norm{E(u)}^2 \big| \mathcal{H}_T \right] du \right] + h_{\max}^{2\delta}t\Gamma_2,
\end{equation}
and again it follows from the tower property that
\begin{equation} 
        \mathbb{E}\Big[ \norm{E(t)}^2  \Big]  \leq \Gamma_1 \int_{0}^{t} \mathbb{E} \Big[ \norm{E(u)}^2  \Big] du  + h_{\max}^{2\delta}t\Gamma_2,\quad t\in[0,T].
\end{equation}
Now we can apply the Gronwall inequality (see Lemma \ref{Gronwall}) to get
\begin{equation} 
    \begin{split}
        \bigg(\mathbb{E}\Big[ \norm{E(t)}^2  \Big] \bigg)^{\frac{1}{2}}  &\leq 
        \bigg( h_{\max}^{2\delta}t\Gamma_2 \exp{\Gamma_1 \int_{0}^{t} dr} \bigg)^{\frac{1}{2}} \\
        &\leq \sqrt{\Gamma_2 t \exp{\Gamma_1 t} } h_{\max}^{\delta}\\
        &\leq C(t)h_{\max}^{\delta},\quad t\in[0,T].
    \end{split}
\end{equation}
Taking the maximum over $t \in [0, T]$ on both sides, the bound \eqref{eq:mainErrorBound} and therefore the statement of the theorem follows. 
\end{proof}

\begin{remark}
     The error constant $C$ in the bound \eqref{eq:mainErrorBound} may depend on the stepsize ratio $\rho$ indirectly through $\Gamma_1$ and $\Gamma_2$. However $C$ has no dependence on the switching rates $\lambda_i$ associated with the Markov chain. This is by contrast with the jump case investigated in \cite{kelly2024strong}, where the error constant depends directly on a Poisson jump intensity rate.
\end{remark}

\section{Numerical Results}
\label{sec:5}

\subsection{Implementation of an adaptive numerical method for SDEwMS} 
\label{subsec:5.1}

First, we simulate the trajectory of a given Markov chain $r$ with state space $\mathbb{S}$ and generator matrix $\Gamma$ (see Eq.~\eqref{generatorMatrix}). For each sampled $\omega \in \Omega$, we wish to compute the value of $r(\tau_i, \omega)$, along with the respective switching times $\tau_i$, for $i = 0, 1, ..., \bar{N}_T$. For a trajectory that has just transitioned to state $S_i$, the holding time before transitioning to a different state is an exponentially distributed random variable with parameter $\lambda_i$, where $\lambda_i = \gamma_{i1} + \gamma_{i2} + ... + \gamma_{iL} - \gamma _{ii}$. The conditional transition probabilities are computed using the entries of $\Gamma$ as 
\[
p_{ij} = \frac{\gamma_{ij}}{\gamma_{i1} +\gamma_{i2} + ... + \gamma_{iL} - \gamma_{ii}}.
\]
Then for each state $i$ we can create a probability mass function with these transition probabilities and sample from it to simulate which state $S_j$ is chosen at each switching time. We store the switching times $(\tau_i)_{i\geq0}$ and the states selected $(S_i)_{i\geq0}$ as two separate vectors for each trajectory $\omega$ to use in the main algorithm.

Then we can begin to construct the mesh, with specific timestepping strategy as described in \eqref{eq:hscheme}. The step size, $h_n$, taken for each step dictates whether we implement the main map $\mathcal{M}$ or the backstop map $\mathcal{B}$ as per \eqref{Y(s)}. In this section we choose as our main map the Milstein map \eqref{eq:MilMap} and as our backstop map the map associated with the implicit Milstein method. Therefore our adaptive method will be strongly convergent in mean-square with order $\gamma=1$.

Running this method for a specified number of trajectories $\omega$ with initial starting point $L_0$, we plot the final values on each trajectory $L_T(\omega)$ as a density histogram.

\subsection{Example: a stochastic model of telomere length with Markovian switching}
Consider the following variant of the SDE \eqref{eq:telomere}
\begin{equation}\label{eq:telomereMS}
dL(t)=-\left(C(t)+A(t)L(t)^2\right)dt+\sqrt{\frac{1}{3}A(t)L(t)^3}dW(t),
\end{equation}
where the random pair $(C(t),A(t))_{t\geq 0}$ evolves according to a Markov chain that is independent of $W$ and with state space  $\mathbb{S}=\{1,2,3,4\}$. Here, each $i\in\mathbb{S}$ corresponds to an element of the set $\mathbb{S}'=\{(c_1,a_1),(c_1,a_2),(c_2,a_1),(c_2,a_2)\}$, where $a_i$ and $c_j$, $i,j=1,2$ are constants calculated from data in~\cite{grasman2011stochastic}. There, different estimation methods provided different values for these constants and so we can consider the use of Markovian switching in this context to be a form of robustness testing. The estimates we will use for our numerical analysis are $(c_1,a_1)=(4.5, 0.22 \times 10^{-6})$, and $(c_2,a_2)=(7.5, 0.41 \times 10^{-6})$, and we will switch between the corresponding states in $\mathbb{S}'$ according to $r$ generated with the algorithm described in Section \ref{subsec:5.1}. For the purposes of our numerical demonstration the generator is chosen to be
\[
\Gamma=\begin{bmatrix}
    -0.3 & 0.1 & 0.1 & 0.1\\
    0.1 & -0.3 & 0.1 & 0.1\\
    0.1 & 0.1 & -0.3 & 0.1\\
    0.1 & 0.1 & 0.1 & -0.3 
\end{bmatrix}
\]
For the adaptive mesh, we set $h_{max}=3 \times 10^{-2}$, $k = 10$ and $\rho=15$ throughout.

\subsection{Numerical Results: Fixed Initial Length $L_0$}

For our fixed initial length we have chosen $L_0 = 1000$ and $T=30$. In Figure \ref{fig:1} we display a histogram of $1000$ values, with $\mathbb{E}[L_{30}]\approx 814.33$ base pairs (bp).
\begin{figure}
    \centering
    \resizebox{0.9\textwidth}{!}
    {\includegraphics{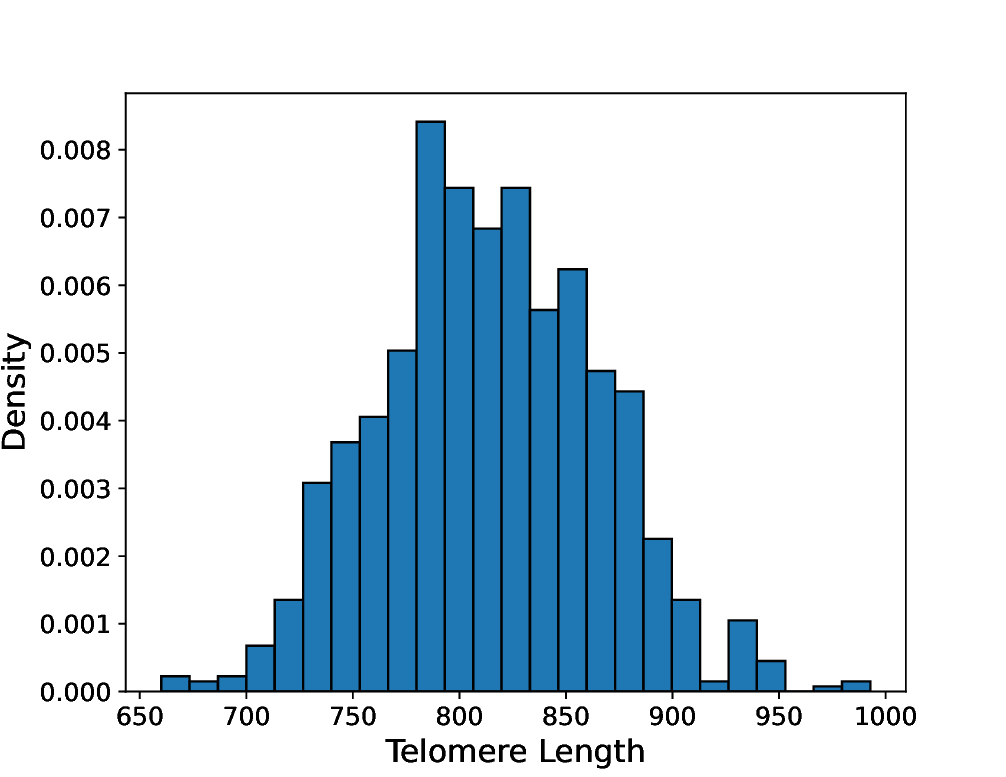}}
    \caption{Histogram of $L_{30}(\omega)$ for 1000 samples.}
    \label{fig:1}
\end{figure}

In Figure \ref{fig:3} we display density histograms produced when the parameters are fixed at $(c,a)=(c_1,a_1)$ (top) and $(c,a)=(c_2,a_2)$ (bottom). The means are computed to be $\mathbb{E}[L_{30}]\approx 862.69$ bp and $\mathbb{E}[L_{30}]\approx 770.75$ bp respectively.
\begin{figure}
    \centering
    \resizebox{0.9\textwidth}{!}{
    \includegraphics{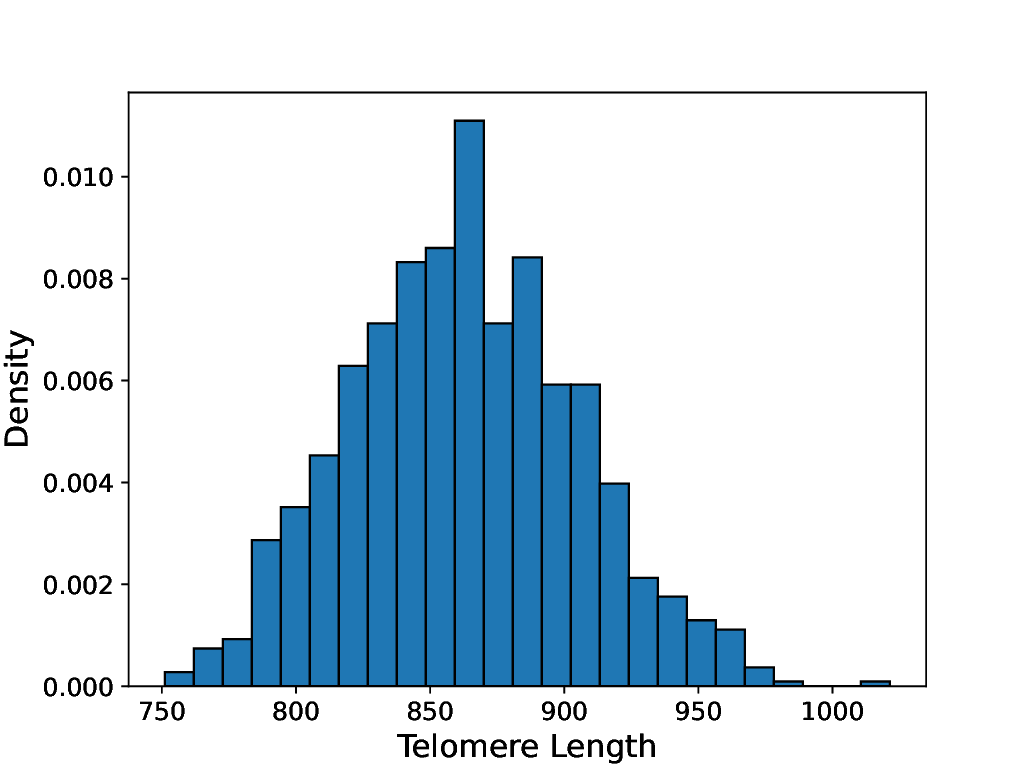}}
    \resizebox{0.9\textwidth}{!}{
    \includegraphics{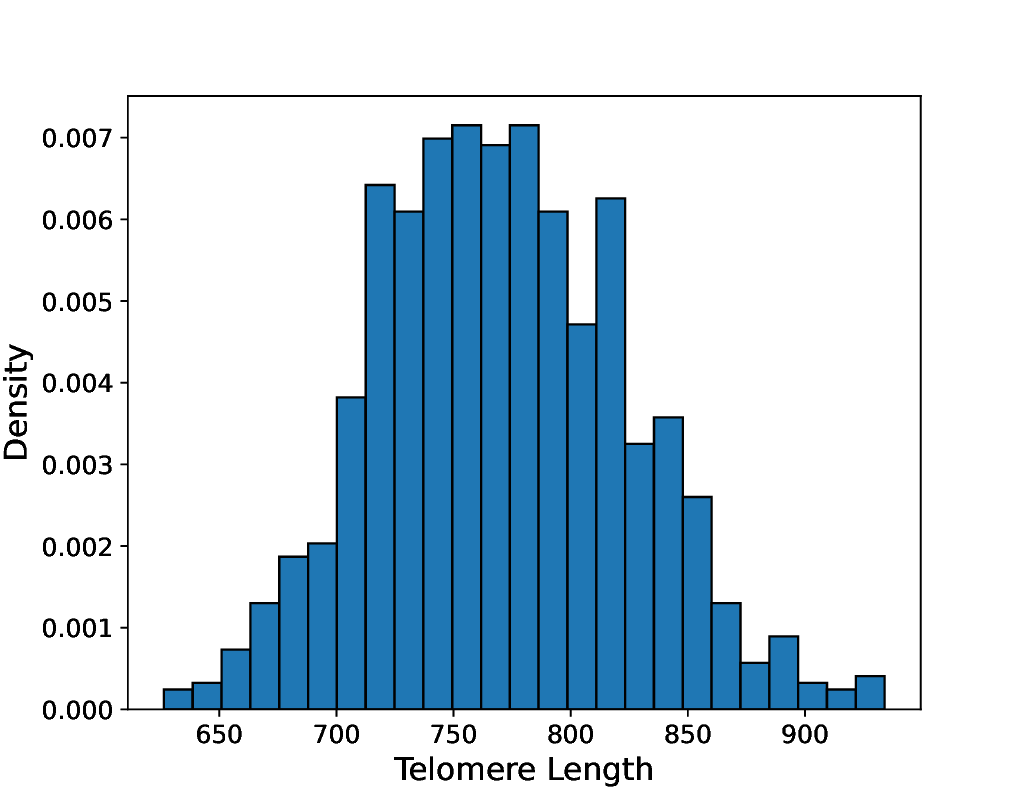}}
    \caption{Histograms of $L_{30}(\omega)$ for 1000 samples for the SDE without switching and choosing $(c,a)=(c_1,a_1)$ (top) and $(c,a)=(c_2,a_2)$ (bottom).}
    \label{fig:3}
\end{figure}

\subsection{Numerical Results: Uniformly Distributed Initial Lengths $L_0$}

Next we randomly generated 1000 initial lengths from a uniform distribution with range $[4000, 8000]$ and ran the algorithm 100 times for each starting length. 
In Figure \ref{fig:4} we display a line graph comparing initial lengths, and for each initial length the final value of a single trajectory as well as the mean value of each end length.  
\begin{figure}
    \centering
    \resizebox{0.9\textwidth}{!}{
    \includegraphics{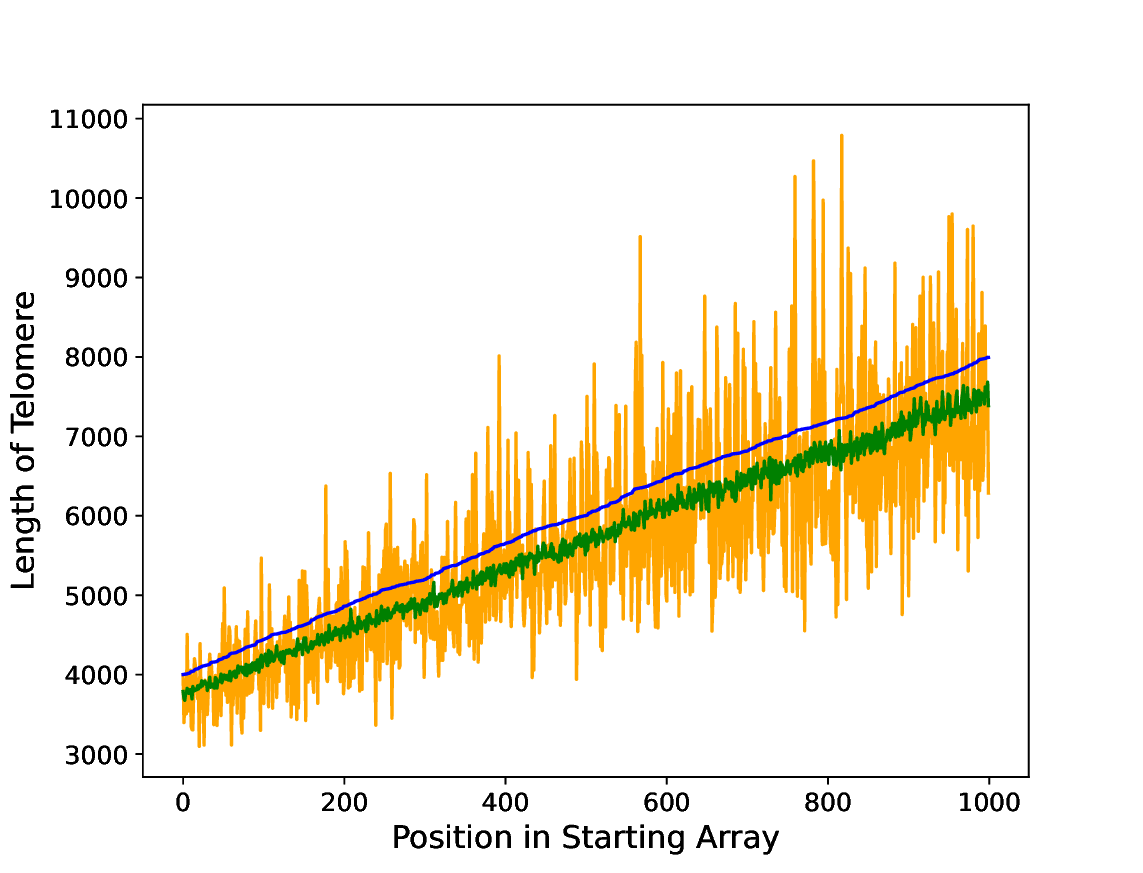}}
    \caption{Initial telomere lengths ordered (blue), actual final length for a single trajectory (orange), and the mean change in length over all trajectories for a given initial value (green).}
    \label{fig:4}
\end{figure}

 And finally in Figure \ref{fig:5} we display a histogram which shows the mean change in length of the telomere. We calculated the mean change for all 1000 initial lengths, over 100 simulations. The mean change observed over all initial lengths and across all simulated trajectories was found to be $\mathbb{E}[L_{5}-L_{30}]\approx-350.74$ bp. The distribution of the simulated data presented here may be compared to the data presented in \cite{grasman2011stochastic}, and we note a qualitative similarity in the mean change histograms. The difference between our estimated mean change and that given in \cite{grasman2011stochastic} ($\mathbb{E}[L_5-L_{30}]\approx-223.68$ bp) may be the result of using a model with Markovian parameters (including potentially the choice of $\Gamma$), as well as possibly reflecting sampling variation and deviations from a uniform distribution of initial values in the real-world data presented in~\cite{grasman2011stochastic}. A more detailed investigation of this is suggested as a future line of research.

\begin{figure}
    \centering
    \resizebox{0.9\textwidth}{!}{
    \includegraphics{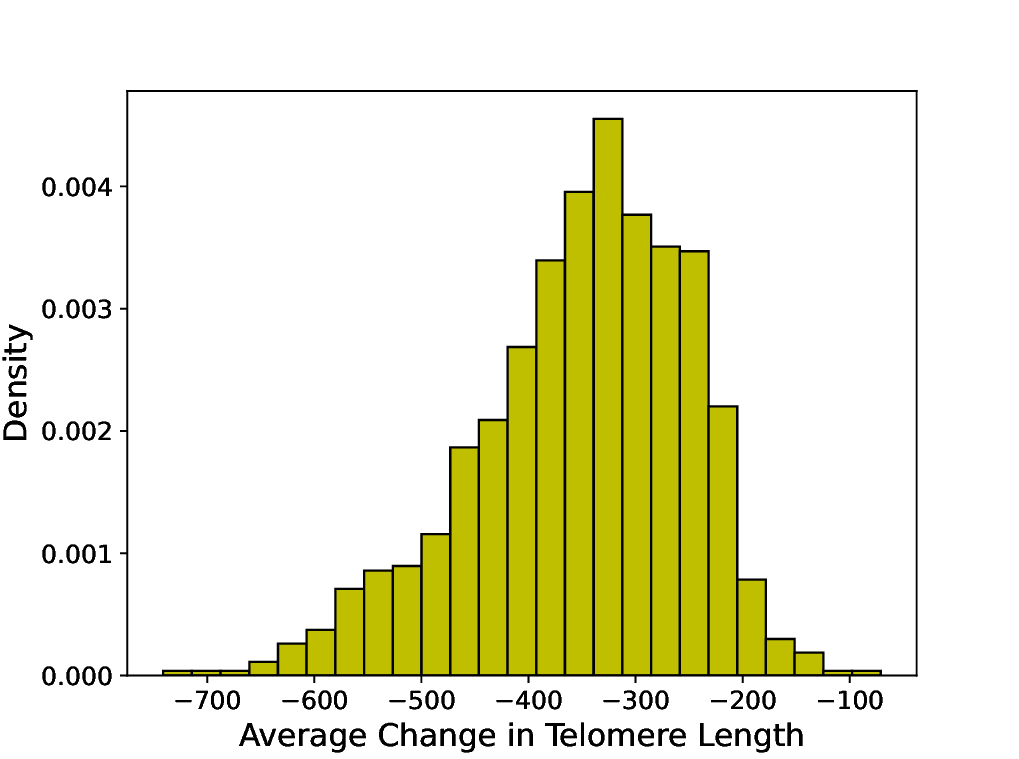}}
    \caption{Density histogram of the mean change in telomere lengths between day $5$ and day $30$.}
    \label{fig:5}
\end{figure}

\medskip
\noindent {\bf Acknowledgement:} CK was supported by the RSE Saltire Research Facilitation Network on Stochastic Differential Equations: Theory, Numerics and Applications (RSE1832).

\bibliographystyle{plain}

\begin{thebibliography}{10}

\bibitem{Beyn2017}
Wolf-J{\"u}rgen Beyn, Elena Isaak, and Raphael Kruse.
\newblock Stochastic {C}-stability and {B}-consistency of explicit and implicit
  {Milstein}-type schemes.
\newblock {\em Journal of Scientific Computing}, 70(3):1042--1077, Mar 2017.

\bibitem{grasman2011stochastic}
J~Grasman, HM~Salomons, and S~Verhulst.
\newblock Stochastic modeling of length-dependent telomere shortening in
  {C}orvus monedula.
\newblock {\em Journal of Theoretical Biology}, 282(1):1--6, 2011.

\bibitem{Guo2018}
Qian Guo, Wei Liu, Xuerong Mao, and Rongxian Yue.
\newblock The truncated {M}ilstein method for stochastic differential equations
  with commutative noise.
\newblock {\em Journal of Computational and Applied Mathematics}, 338:298 --
  310, 2018.

\bibitem{Desmond}
Desmond~J. Higham, Xuerong Mao, and Lukasz Szpruch.
\newblock Convergence, non-negativity and stability of a new {M}ilstein scheme
  with applications to finance.
\newblock {\em Discrete and Continuous Dynamical Systems - B},
  18(8):2083--2100, 2013.

\bibitem{HJK2012}
M.~Hutzenthaler, A.~Jentzen, and P.~E. Kloeden.
\newblock Strong convergence of an explicit numerical method for {SDE}s with
  nonglobally {L}ipschitz continuous coefficients.
\newblock {\em Annals of Applied Probability}, 22:1611--1641, 2012.

\bibitem{hutzenthaler2011strong}
Martin Hutzenthaler, Arnulf Jentzen, and Peter~E Kloeden.
\newblock Strong and weak divergence in finite time of {Euler}'s method for
  stochastic differential equations with non-globally {Lipschitz} continuous
  coefficients.
\newblock {\em Proceedings of the Royal Society of London A: Mathematical,
  Physical and Engineering Sciences}, 467(2130):1563--1576, 2011.

\bibitem{KLSBit}
C\'onall Kelly, Gabriel Lord, and Fandi Sun.
\newblock Strong convergence of an adaptive time-stepping {Milstein} method for
  {SDEs} with monotone coefficients.
\newblock {\em BIT Numerical Mathematics}, 63(33), 2023.

\bibitem{kelly2024strong}
C\'onall Kelly, Gabriel~J. Lord, and Fandi Sun.
\newblock Strong convergence of a class of adaptive numerical methods for
  {SDE}s with jumps.
\newblock {\em Mathematics and Computers in Simulation}, 227:461--476, 2025.

\bibitem{kelly2018adaptive}
C{\'o}nall Kelly and Gabriel~J. Lord.
\newblock Adaptive time-stepping strategies for nonlinear stochastic systems.
\newblock {\em IMA Journal of Numerical Analysis}, 38(3):1523--1549, 2018.

\bibitem{KLS2023}
C\'onall Kelly, Gabriel~J. Lord, and Fandi Sun.
\newblock Strong convergence of an adaptive time-stepping {M}ilstein method for
  {SDE}s with monotone coefficients.
\newblock {\em {BIT} Numerical Mathematics}, 63(33), 2023.

\bibitem{mao2006stochastic}
X.~Mao and C.~Yuan.
\newblock {\em Stochastic Differential Equations with Markovian Switching}.
\newblock G - Reference, Information and Interdisciplinary Subjects Series.
  Imperial College Press, 2006.

\bibitem{mao2007SDEapp}
Xuerong Mao.
\newblock {\em Stochastic differential equations and applications}.
\newblock Woodhead Publishing, Cambridge, 2 edition, 2007.

\bibitem{mao2015truncated}
Xuerong Mao.
\newblock The truncated {E}uler--{M}aruyama method for stochastic differential
  equations.
\newblock {\em Journal of Computational and Applied Mathematics}, 290:370--384,
  2015.

\bibitem{Norris_1997}
J.~R. Norris.
\newblock {\em Markov Chains}.
\newblock Cambridge Series in Statistical and Probabilistic Mathematics.
  Cambridge University Press, 1997.

\bibitem{wang2013tamed}
Xiaojie Wang and Siqing Gan.
\newblock The tamed {Milstein} method for commutative stochastic differential
  equations with non-globally {Lipschitz} continuous coefficients.
\newblock {\em Journal of Difference Equations and Applications},
  19(3):466--490, 2013.

\end{thebibliography}

\end{document}